\documentclass[smallextended]{amsart}
\usepackage[margin=1.5in]{geometry}
\usepackage{amsthm}
\usepackage{mathtools,mathrsfs}
\usepackage{thmtools}
\usepackage{color}
\usepackage{graphicx}
\usepackage[utf8]{inputenc}
\usepackage{amssymb}
\usepackage{hyperref}

\hypersetup{colorlinks, citecolor=blue}
\usepackage{url}
\usepackage[numbers]{natbib}
\newtheoremstyle{mydef}
{\topsep}{\topsep}%
{}{}%
{\itshape}{}
{\newline}
{%
  \rule{\textwidth}{0.0pt}\\*%
  \thmname{#1}~\thmnumber{#2}\thmnote{\-\ #3}.\\*[-1.0ex]%
  \rule{\textwidth}{0.0pt}}%

\DeclarePairedDelimiter\ceil{\lceil}{\rceil}

\newtheorem{theorem}{Theorem}[section]
\newtheorem{remark}{Remark}
\newtheorem{definition}{Definition}
\newtheorem{lemma}{Lemma}
\newtheorem{corollary}{Corollary}[theorem]
\numberwithin{equation}{section}

\renewcommand{\a}{\alpha}
\newcommand{\e}{\epsilon}
\renewcommand{\t}{\theta}
\renewcommand{\o}{\eta}
\newcommand{\G}{\Gamma}
\newcommand{\m}{\mathbb{R}}
\newcommand{\C}{\mathbb{C}}
\newcommand{\R}{\mathcal{R}}

\newcommand{\E}{\mathbb{E}}
\newcommand{\Q}{\mathcal{Q}}
\renewcommand{\S}{\Sigma}

\title[Bounds on the Phillips calculus]{Bounds on the Phillips calculus of abstract first order differential operators}
\author{Himani Sharma}
\thanks{Email ID: himani.sharma@anu.edu.au, Department Address: Mathematical Sciences Institute, Australian National University, Canberra ACT-2600, Australia. ORCID-0000-0001-8248-1044
}
\begin{document}

\maketitle

\begin{abstract}
For an operator generating a group on $L^p$ spaces transference results give bounds on the Phillips functional calculus also known as spectral multiplier estimates. In this paper we consider specific group generators which are abstraction of first order differential operators and prove similar spectral multiplier estimates assuming only that the group is bounded on $L^2$ rather than $L^p$. We also prove an R-bounded H{\"o}rmander calculus result by assuming an abstract Sobolev embedding property and show that the square of a perturbed Hodge-Dirac operator has such calculus.\\

\end{abstract}

\begingroup
\allowdisplaybreaks
\section{Introduction}
Let $X$ be a Banach space and $iD$ generate a bounded $C_0$-group $(U(t))_{t\in\m}$ on $X$. For $g\in L^1(\m)$ we define
\begin{equation*}
    \Phi_g(D)x:=\int_{-\infty}^\infty g(t)U(t)x dt, \;x\in X.
\end{equation*}
The map $(g\mapsto\Phi_g(D)):L^1(\m)\to B(X)$, where $B(X)$ is the set of all bounded operators on $X$ and $\Phi_g$ is obtained using the inverse Fourier-Stieltjes transform, is called the Phillips calculus. It is immediate that the integral exists and 
\begin{equation*}
    ||\Phi_g(D)||_{B(X)}\leq ||g||_{L^1}\sup_{t\in\m}||U(t)||_{B(X)}.
\end{equation*}
The idea behind getting a bound on the Phillips calculus of some operators is to achieve such estimates for function classes other than $L^1$.
The transference principle introduced by Cofiman and Weiss in \cite{CW} gives a great functional calculus result for such group generators. They proved that if $g\in L^1(\m)$ and $K_g:L^p(\m,X)\to L^p(\m,X)$ for $p\in[1,\infty)$ is defined as 
\begin{equation*}
    K_gf(t):=\int_{-\infty}^\infty g(s)f(t-s)ds,\;f\in L^p(\m,X)
\end{equation*}
then, 
\begin{equation*}
    ||\Phi_g(D)||_{B(X)}\leq ||K_g||_{B(L^p(\m,X))}\sup_{t\in\m}||U(t)||^2_{B(X)}
\end{equation*}
Inspired by them, Haase later in 2009 (see Theorem 3.2,\cite{HaM09}), where the technique for the proof of this theorem is explained very well in (Section 3,\cite{HaM11}) proved a transference principle where the group can grow exponentially. In this case the functions $g$ need to have holomorphic extensions to a strip of the complex plane. Kriegler in 2012 (Theorem 4.9, \cite{CK}) proved a similar result for a group that grows at most polynomially on $L^p$.  In the case of both Coifman-Weiss and Kriegler's results the functions can be compactly supported. In fact, the class of functions they consider is a variation of the H{\"o}rmander class and involves functions which are differentiable sufficiently many times and whose derivatives have appropriate decay. Since the class of smooth and compactly supported functions is not contained in $H^\infty$, therefore to obtain a calculus that works for such functions we need a result similar to Coifman-Weiss or Kriegler's result. Kriegler and Weis (Theorem 1.1, \cite{KL}) have given a very general approach in the case of semigroup generators and proved that these operators have H{\"o}rmander calculus but their result has a little weakness as their R-boundedness assumption is quite difficult to check.

This paper deals with the case when there is no available $L^p$ information (which would be required to use transference) on the group and the only information given to us is on $L^2$, that is, the group generator is either self-adjoint or similar to a self-adjoint operator. The question then arises: ``what extra properties do we need on the group or its generator?" Notice that some of the transference results that we have talked about before transfer to the derivative despite the fact that a general group generator does not have to be a differential operator whatsoever. Thus, it will turn out that our framework is well suited for (Hodge-Dirac) differential operators. The question is how are we going to exploit this structure, given only $L^2$ information? We show in this paper that if we know some $L^p-L^2$ mapping properties (which look like Sobolev embeddings and are likely to happen when the operator is differential) then we get something which is very reminiscent of Kriegler's result.

Chen, Ouhabaz, Sikora and Yan in \cite{COSY} considered second order differential operators acting on $L^p$ spaces over a doubling measure metric space and exploited certain properties including the finite-speed propagation property to obtain sharp spectral multiplier results. What we have done is similar but slightly more abstract. In fact, our results are in the same spirit, but for slightly more specific operators, as the result in their recent paper \cite{COSY20}. For instance, the H{\"o}rmander calculus result that we have obtained in our Theorem \ref{Thm3.3} for the  square of the Hodge-Dirac operator $\Pi_B$ (defined in preliminaries) implies their main result (Theorem 3.1, \cite{COSY20}) for uniformly elliptic operators in divergence form (for example, $divA\nabla$, $A\in L^\infty$) and generalises it to all of the squares of Dirac operators defined in \cite{AKM}, by Axelsson, Keith and McIntosh. Our result thus complements \cite{COSY20} by giving a different approach and substantially simpler proof. The motivation is the same: to develop techniques that are abstract enough to be applied in a wide range of settings, yet concrete enough to give strong results for differential operators.

It is worth noticing that given a self-adjoint differential operator on $L^2$ it is generally difficult to decide whether or not it generates a group on $L^p$. For instance, the first derivative in $L^p(\m)$ generates the group of translation but $i$ times the second derivative generates a group only if $p=2$. See \cite{Frey2020} by Frey and Portal for fairly specific examples of groups generated by operators with rough coefficients in $L^p$. Exponential growth of a group can be obtained by perturbation but polynomial boundedness is quite difficult to establish as can be seen in the paper by Rozendaal and Veraar (Theorem 1.1, \cite{RV}). Coifman-Weiss and Kriegler's transference results thus have assumptions that are difficult to establish while general transference results do not give a large functional calculus. We present an intermediate result suited to differential operators.

In addition to this, we have exploited the case where we take our operator to be a perturbed Hodge-Dirac operator as described in the paper \cite{FMP} by Frey, McIntosh and Portal. For these operators, we prove a much stronger result which shows that their square has R-bounded H{\"o}rmander calculus on $L^p$ spaces for $p$ in some interval. This is possible by combining the results of Frey, McIntosh, Portal \cite{FMP} with that of Kriegler-Weis (see Theorem \ref{Thm2.3}) and Kunstmann (see Theorem \ref{Thm2.2}) where an off diagonal assumption gives us the appropriate R-bound.

\subsection{Acknowledgement}
I am extremely thankful to my supervisor Pierre Portal for his continuous support and encouragement. It was his perseverance and foresight which made me kept working on this problem even after facing so many ups and downs. I am also very grateful to Dorothee Frey for her valuable suggestion to Pierre to look at the paper of Peer Kunstmann. I would like to express my heartfelt thanks to Christoph Kriegler as well for suggesting changes in our Theorem 3.2 and improving the result.
\subsection{Declarations}
\begin{itemize}
    \item \textbf{Funding}: This research is supported by the Australian Government Research Training Program (AGRTP) scholarship.
    \item \textbf{Conflict of Interest}-Not applicable
    \item \textbf{Availability of data and material}-Not applicable
    \item \textbf{Code availability}-Not applicable
\end{itemize}

\section{Preliminaries}
This section consists of the notations used in this paper and some definitions and auxiliary results required to prove the main theorems.
\subsection{Notations}
Henceforth, we fix two positive natural numbers $d$ and $n$ and denote inequalities ``up to a constant" between two positive quantities $x,y$ by $x\lesssim y$. By this we mean that there exists a constant $C>0$, independent of all relevant quantities in the statement, such that $x\leq Cy$. We denote equivalence ``up to a constant" by $x\simeq y$ which means that there exist constants $c$ and $c'$ such that $\frac{1}{c}x\leq y\leq c'x$. \\
The set of all bounded linear operators on $X$ is denoted by $B(X)$ and we write $\m^*=\m\setminus\{0\}$.\\
In this paper we will be taking all our function spaces to be $\C^n$-valued without mentioning the explicit value of $n$ which might change from one theorem to another. In fact, all the results and definitions that we are using, for instance, Lemma \ref{Lem 1},Theorem \ref{Thm2.2} and Theorem \ref{Thm2.3} are valid for $\C^n$-valued $L^p$ spaces.\\
For $p\in(1,\infty)$ and an unbounded operator $D$ on $L^p(\m^d;\C^n)$, we denote by $\mathcal{D}_p(D),$ $\R_p(D),$ $\mathcal{N}_p(D)$ its domain, range and null space, respectively.\\
We use $p_*$ to denote Sobolev exponents below $p$. That is, for $p\in[1,\infty)$, $p_*=\frac{dp}{d+p}$.
Since $2_*$ denotes one Sobolev exponent below 2 similarly $2_{**}$ denotes two Sobolev exponent below 2. 
We use the notation $2_{(*)m}$ to denote the $m^{th}$ Sobolev exponent below 2 where $m\in\mathbb{N}$.

\subsection{Definitions and some known results}
In this subsection we are focusing on some fundamental results and definitions needed in the main body of the paper. 

This paper deals with self-adjoint first order differential operators which have a holomorphic functional calculus (also called $H^\infty$-calculus). On $L^p$ the main results concerning this calculus have been developed in \cite{MR1481596,MR1364554,MR912940} and for more information about them one can refer to the lecture notes \cite{MR1394696,MR2108959} and book \cite{MR2244037}.

\begin{definition}[\textbf{Finite propagation speed}]
Let $D$ be a self-adjoint operator on $L^2(\m^d;\C^n)$. Then $e^{itD}$ is said to have finite propagation speed $\kappa_D\leq \kappa$ if 
\begin{equation*}
    \mathrm{supp}(e^{itD}u)\subset K_{\kappa|t|}:=\{x\in \m^d;\; \mathrm{dist}(x,K)\leq\kappa|t|\}\;\;\;\forall t\in\m
\end{equation*}
whenever $\mathrm{supp}(u)\subset K\subset \m^d$, for some compact set $K$. Here $\mathrm{dist}(x,K):=\inf \{d(x,y): y\in K\}$
\end{definition}

\begin{definition}[\textbf{Off-diagonal bounds}]
Let $p\in[1,2]$. A family of operators $\{T_t:t\in \mathcal{B}\}\subset B(L^2(\m^d;\C^n))$, where $\mathcal{B}\subseteq\m^*$, is said to have $L^p-L^2$ off-diagonal estimates of order $M$ if there exists a constant $C_M>0$ such that for all $t\in\mathcal{B}$, all Borel sets $E,F\subseteq\m^d$ and all $u\in L^p(\m^d;\C^n)$ with support in $F$ we have
\begin{equation*}
    ||T_tu||_{L^2(E)}\leq C_M|t|^{-d\big(\frac{1}{p}-\frac{1}{2}\big)}\bigg(1+\frac{d(E,F)}{t}\bigg)^{-M}||u||_{L^p(F)}
\end{equation*}
where $d(E,F)=\inf\{|x-y|:x\in E,y\in F\}$.
\end{definition}

\begin{lemma}[Lemma 4.3, \cite{PA}]\label{Lem 1}
Let $p\in[1,2]$ and $T$ be a linear operator which satisfies $L^p-L^2$ off-diagonal estimates in the form $||Tu||_{L^2(E)}\leq g(d(E,F))\cdot$ $||u||_{L^p(F)}$ where $E,F$ are closed cubes with $\mathrm{supp}(u)\subseteq F$ and $g$ is some function such that $s^{d|\frac{1}{p}-\frac{1}{2}|}\sum_{k\in\mathbb{Z}^d}g(\sup(|k|-1,0)s)$ is finite for any $s>0$. Then $T$ is a bounded operator on $L^p(\m^d)$.
\end{lemma}

We shall now recall the following definition and result of operators acting on tent spaces which is of great importance in proving Theorem \ref{Thm3.2}.

\begin{definition}[\textbf{Tent spaces}]
The tent space $T^{p,2}(\m_+^{d+1};\C^n)$, first introduced by Coifman, Meyer, and Stein in \cite{MR791851}, is defined as the completion of $C_c^\infty(\m_+^{d+1};\C^n)$ with respect to the following norm for $p\in[1,\infty)$
\begin{align*}
    ||F||_{T^{p,2}}&=\bigg(\int_{\m^d}\bigg(\int_0^\infty\int_{B(x,t)}|F(t,y)|^2\frac{dydt}{t^{d+1}}\bigg)^{\frac{p}{2}}dx\bigg)^{\frac{1}{p}},
\end{align*}
where $B(x,t)$ denote balls in $\m^{d+1}_+$ with center at $x$ and radius $t$.
\end{definition}

\begin{lemma}[Lemma 2.9, \cite{FMP}]\label{Lem 2}
Let $p\in(1,\infty)$ and let $\{U_t\}_{t>0}$ be a family of operators on $L^2(\m^d;\C^n)$ with $L^2-L^2$ off diagonal bounds of order $M>\frac{d}{\min\{p,2\}}$. Then there exists $C>0$ such that for all $F\in T^{p,2}(\m_+^{d+1};\C^n)$
\begin{equation*}
    ||(t,x)\mapsto U_tF(t,.)(x)||_{T^{p,2}}\leq C||F||_{T^{p,2}}.
\end{equation*}
\end{lemma}

\begin{definition}
Let $0\leq\t<\mu<\frac{\pi}{2}$. Define closed and open sectors and double sectors in the complex plane by
\begin{align*}
    \S_{\t+}&:=\{z\in\C:|\arg z|\leq\t\}\cup\{0\},\;\S_{\t-}:=-\S_{\t+},\\
    \S^0_{\mu+}&:=\{z\in\C:z\neq0,|\arg z|<\mu\},\;\S^0_{\mu-}:=-\S^0_{\mu+},\\
    \S_\t&:=\S_{\t+}\cup \S_{\t-},\; \S^0_\mu:=\S^0_{\mu+}\cup \S^0_{\mu-}.
\end{align*}
\end{definition}
\noindent
Denote by $H(\S^0_\mu)$ the space of all holomorphic functions on $\S^0_\mu$. Let further
\begin{equation*}
    \Psi_a^b(\S_\mu^0):=\{\psi\in H(\S^0_\mu):\exists C>0:|\psi(z)|\leq C|z|^a(1+|z|^{a+b})^{-1}\forall z\in\S^0_\mu\}
\end{equation*}
for every $a,b>0$, and set $\Psi(\S^0_\mu):=\bigcup_{a,b>0}\Psi^b_a(\S^0_\mu)$. We say that $\psi\in\Psi(\S^0_\mu)$ is non-degenerate if neither of the restrictions $\psi|_{\S^0_{\mu\pm}}$ vanishes identically.\\

Recall now the main definition and results related to Hardy spaces associated with bisectorial operators and refer to \cite{MR2365673,MR2163867,MR2448989} and the references therein for detailed description.

\begin{definition}[\textbf{Hardy spaces associated with bisectorial operators}]\label{Def1}
Let $0\leq\theta<\mu<\frac{\pi}{2}$ and $D$ be a $\theta$- bisectorial operator on $L^2(\m^d,\mathbb{C}^n)$ such that $\{(I+itD)^{-1};t\in\m\setminus\{0\}\}$ has $L^2-L^2$ off diagonal bounds of order $M>\frac{d}{2}$. Suppose that $D$ has a bounded $H^\infty$ functional calculus of angle $\omega\in(\theta,\mu)$ on $L^2$. Then for $p\in[1,\infty)$ and non-degenerate $\psi\in\Psi(\S^0_\mu)$, the Hardy space $H^p_{D,\psi}(\m^d,\C^n)$ associated with $D$ and $\psi$ is defined as the completion of the space
\begin{equation*}
    \{u\in\overline{\R_2(D)}:\Q_\psi u\in T^{p,2}(\m_+^{d+1},\mathbb{C}^n)\}
\end{equation*}
with respect to the norm
\begin{equation*}
    ||u||_{H^p_{D,\psi}}:=||\Q_\psi u||_{T^{p,2}}
\end{equation*}
where $\Q_\psi u(x,t):=\psi(tD)u(x),\text{ for }u\in L^2(\m^d,\mathbb{C}^n)\text{ and }t>0$.
\end{definition}

\begin{theorem}[Theorem 7.10, \cite{MR2448989}]\label{Thm2.1}
Recall that $D$ is bisectorial here. Let $\e>0$. Let $p\in(1,2]$ and $\phi,\psi\in\Psi_\e^{\frac{d}{2}+\epsilon}(\S_\mu^0)$, or $p\in[2,\infty)$ and $\phi,\psi\in\Psi_{\frac{d}{2}+\e}^\e(\S^0_\mu)$, where $\mu>\theta$ and both $\phi$ and $\psi$ are non-degenerate. Then,
\begin{itemize}
    \item[(1)] $H^p_{D,\phi}(\m^d,\C^n)=H^p_{D,\psi}(\m^d,\C^n)=:H^p_D(\m^d,\C^n)$;
    \item[(2)] For all $u\in\Psi(\S^0_\mu)$ and all $u\in H^p_D(\m^d,\C^n)$, we have
        $$||(t,x)\mapsto\psi(tD)f(D)u(x)||_{T^{p,2}}\lesssim||f||_\infty||u||_{H^p_D}.$$
\end{itemize}
In particular, $D$ has a bounded $H^\infty$ functional calculus on $H^p_D(\m^d,\C^n).$
\end{theorem}

\begin{definition}[$R$-boundedness]
Let $\mathcal{T}$ be a subset of $B(X,Y)$ where $X$ and $Y$ are Banach spaces. We say that $\mathcal{T}$ is R-bounded if there exists a $C<\infty$ such that
\begin{equation*}
    \E\bigg|\bigg|\sum_{k=1}^n\e_kT_kx_k\bigg|\bigg|\leq C\E\bigg|\bigg|\sum_{k=1}^n\e_kx_k\bigg|\bigg|
\end{equation*}
for any $n\in\mathbb{N},\;T_1,T_2,...,T_n\in\mathcal{T}$ and $x_1,x_2,...,x_n\in X$. The smallest admissible constant is denoted by $\mathscr{R}(\mathcal{T})$. The Rademacher sequence $(\e_k)_k$ is i.i.d., and satisfies $P(\e_k=1)=P(\e_k=-1)=\frac{1}{2}$.
\end{definition}

\begin{definition}[$R_s$-boundedness]
Let $s\in[1,\infty]$. A subset $\mathcal{T}$ of bounded (sub)linear operators on $L^p(\Omega)$ is called $R_s$-bounded in $L^p(\Omega)$ if there exists a constant $C>0$ such that for all finite families $(T_k)$ in $\mathcal{T}$ and $(f_k)$ in $L^p(\Omega)$, we have
\begin{align*}
    \big|\big|\big(\sum_k|T_kf_k|^s\big)^{\frac{1}{s}}\big|\big|_p&\leq C\big|\big|\big(\sum_k|f_k|^s\big)^{\frac{1}{s}}\big|\big|_p,\;\text{if } 1\leq s<\infty\\
    ||\sup_{k}|T_kf_k|\;||_p&\leq C||\sup_k|f_k|\;||_p,\;\text{if }s=\infty.
\end{align*}
The infimum of all such $C$ is denoted by $\mathscr{R}_s(\mathcal{T};L^p)$.
\end{definition}

Note: For $s=2$, $R_s$-boundedness is $R$-boundedness.

\begin{theorem}[Theorem 2.2, \cite{PC}]\label{Thm2.2}
Let $(\Omega,N,|\cdot|)$ be a space of homogeneous type such that $B(x,\lambda\rho)\leq c_\Omega\lambda^N|B(x,\rho)|$, $\rho>0,\lambda\geq1,x\in\Omega.$ Let $1\leq p\leq q\leq\infty$, and assume that $(S(t))_{t\in\tau}$ is a family of linear operators on $L^p\cap L^q$ such that
\begin{equation*}
    ||1_{B(x,\rho(t))}S(t)1_{A(x,\rho(t),k)}||_{p\to q}\leq |B(x,\rho(t))|^{-(\frac{1}{p}-\frac{1}{q})}h(k),\;t\in\tau,x\in\Omega,k\in\mathbb{N}_0,
\end{equation*}
where $\rho:\tau\to(0,\infty)$ is a function and the sequence $(h(k))_{k\in\mathbb{N}_0}$ satisfies $h(k)\leq c_\delta(k+1)^{-\delta}$ for some $\delta>\frac{d}{p}+\frac{1}{p'}$. Then we have
\begin{equation*}
  \{S(t):t\in\tau\} \text{ is } R_s\text{- bounded in } L^r(\Omega)
\end{equation*}
for all $(r,s)\in(p,q)\times[p,q]\cup\{(p,p),(q,q)\}$.
\end{theorem}

\noindent
We shall now briefly introduce the H{\"o}rmander functional calculus and the result on it given by Kriegler and Weis in \cite{KL}.

The classical spectral multiplier theorem of Mikhlin- H{\"o}rmander shows that if $f$ is a bounded function defined on $(0,\infty)$ and $u(f)=f(-\Delta)$ is an operator on $L^p(\m^d)$ defined as $\mathcal{F}[u(f)g](\xi)=f(|\xi|^2)\mathcal{F}[g](\xi)$, then $u(f)$ is bounded on $L^p(\m^d)$ for $p\in(1,\infty)$ provided $f$ satisfies
\begin{equation}\label{eq1}
    \underset{R>0}{\sup}\int_{R/2}^{2R}\Big|t^k\Big(\frac{d}{dt}\Big)^{(k)}f(t)\Big|^2\frac{dt}{t}<\infty\;\;\;(k=0,1,...,\a),\;\a=\ceil*{d/2}.
\end{equation}
Let us denote by $\mathcal{H}^\a_2$ the ``H{\"o}rmander class" of all functions $f$ satisfying \eqref{eq1}. Then by the above Fourier multiplier theorem
\begin{equation*}
    f\in\mathcal{H}^\a_2\mapsto f(-\Delta)\in B(L^p(\m^d)).
\end{equation*}
Let $\alpha>\frac{1}{2}$. We can also define the $\mathcal{H}^\a_2$ class as
\begin{equation*}
    \mathcal{H}^\a_2=\{f\in L^2_{loc}(\m_+):||f||_{\mathcal{H}^\a_2}=\underset{t>0}{\sup}||\phi f(t\cdot)||_{W^\a_2(\m)}<\infty\},
\end{equation*}
where $\phi$ is a non-zero $C_c^\infty(0,\infty)$ function (different choices resulting in equivalent norms) and $W^\a_2(\m)$ stands for the usual Sobolev space. We say that an operator $D$ defined on a Banach space $X$ has a H{\"o}rmander calculus if for some $\theta\in(0,\pi)$ and any $f\in H^\infty(\Sigma_\theta)$
\begin{equation*}
||f(D)||_{B(X)}\lesssim ||f||_{\mathcal{H}^\a_2}.
\end{equation*}

\begin{theorem}[Theorem 1.1, \cite{KL}]\label{Thm2.3}
Let $A$ be a 0-sectorial operator on a space $L^p(U), 1<p<\infty$, where $U$ is an open subset of $\m^d$. Suppose further that $A$ has a bounded $H^\infty(\S_\t)$ calculus for some $\t\in(0,\frac{\pi}{2})$. Then R-boundedness of the set 
\begin{equation*}
    (S)_\a:=\bigg\{\bigg(\frac{\pi}{2}-|\t|\bigg)^\a\exp(-te^{i\t}A): t>0,\t\in\bigg(\frac{-\pi}{2},\frac{\pi}{2}\bigg)\bigg\}
\end{equation*}
implies that $A$ has an R-bounded $\mathcal{H}^\gamma_2$ calculus for $\gamma>\a+\frac{1}{2}$, that is, the set
\begin{equation*}
    \Big\{f(A):||f||_{\mathcal{H}^\gamma_2}\leq1\Big\} \; \text{ is R-bounded}.
\end{equation*}
\end{theorem}

\noindent

As a specific case of a first order differential operator we have considered perturbed Hodge-Dirac operators of the form $\Pi_B=\Gamma+\Gamma^*_B=\Gamma+B_1\Gamma^*B_2$, where $\Pi=\Gamma+\Gamma^*$ is a Hodge-Dirac operator. One simple example of it is given by taking a $2\times2$ matrix
   $\Pi= \begin{pmatrix}
    0 & -div\\
    \nabla & 0
    \end{pmatrix}$ defined on $L^2(\m^d;\C)\oplus L^2(\m^d;\C^n)$
where, $\Gamma= \begin{pmatrix}
0&-div\\
0&0
\end{pmatrix}$ and $\Gamma^*=\begin{pmatrix}
0&0\\
\nabla&0
\end{pmatrix}$ .
The type of $\Pi_B$ we need on $L^2(\m^d;\C^n)$ is defined and explained below. We refer to the papers \cite{AKM,PTM,FMP} for more details on this operator. See Corollary \ref{Cor1} and Corollary \ref{Cor2} for results on it. 

\begin{definition}[\textbf{Perturbed Hodge-Dirac operators}]\label{Def8}
A perturbed Hodge-Dirac operator defined on $L^2(\m^d;\C^n)$ is an operator of the form
\begin{equation*}
    \Pi_B:=\G+\G_B^*:=\G+B_1\G^*B_2,
\end{equation*}
where $\Pi=\G+\G^*$ is a Hodge-Dirac operator with constant coefficients, and $B_1,B_2$ are multiplication operators by $L^\infty(\m^d,B(\C^n))$ functions which satisfy
\begin{align*}
    \G^*B_2B_1\G^*&=0,\text{ in the sense that }\mathcal{R}_2(B_2B_1\G^*)\subset\mathcal{N}_2(\G^*);\\
    \G B_1B_2\G&=0,\text{ in the sense that }\mathcal{R}_2(B_1B_2\G)\subset\mathcal{N}_2(\G);
\end{align*}
\begin{align*}
    Re(B_1\G^*u,\G^*u)&\geq\kappa_1||\G^*u||_2^2,\;\forall u\in\mathcal{D}_2(\G^*),\text{ and }\\
    Re(B_2\G u,\G u)&\geq \kappa_2||\G u||_2^2,\;\forall u\in\mathcal{D}_2(\G)
\end{align*}
for some $\kappa_1,\kappa_2>0.$
\end{definition}
\noindent
Such operators satisfy the following invertibility properties, where $\frac{1}{p}+\frac{1}{p'}=1$
\begin{equation*}
    (\Pi_B(p))\;\;\;\;||u||_p\leq C_p||B_1u||_p\;\forall u\in\overline{\mathcal{R}_p(\G^*)}\;\text{  and  }\;||v||_{p'}\leq C_{p'}||B_2^*v||_{p'}\;\forall v\in\overline{\mathcal{R}_{p'}(\G)}
\end{equation*}
when $p=2$. A perturbed Hodge-Dirac operator $\Pi_B$ Hodge decomposes $L^p(\m^d;\C^n)$ for some $p\in(1,\infty)$, if $(\Pi_B(p))$ holds and there is a splitting into complemented subspaces
\begin{equation*}
    L^p(\m^d;\C^n)=\mathcal{N}_p(\Pi_B)\oplus\overline{\mathcal{R}_p(\Pi_B)}=\mathcal{N}_p(\Pi_B)\oplus\overline{\mathcal{R}_p(\G)}\oplus\overline{\mathcal{R}_p(\G^*)}.
\end{equation*}
It is proved in (Proposition 2.2,\cite{AKM}) for $p=2$. Here we consider the open interval $(p_H,p^H)$, where $1\leq p_H<2<p^H\leq\infty$ (see Proposition 2.17,\cite{FMP}) on which $\Pi_B$ Hodge decomposes $L^p$. Such $p_H<2<p^H$ always exist.\\

One of the important and very useful results that Frey, McIntosh and Portal have obtained in \cite{FMP} is that the Hardy space $H^p_{\Pi_B}$ associated with the perturbed Hodge-Dirac operator $\Pi_B$ coincides with the $L^p$ closure of the $\mathcal{R}_p(\Pi_B)$ whenever $p\in(p_H,p^H)$.

\section{Main Results}

The following lemma is very useful in proving our first theorem. It can be seen that the idea to have an assumption on the resolvent operator, to be bounded from $L^{p_*}-L^p$, in this lemma comes from [Lemma 7.1, \cite{FMP}] for the operator $\Pi_B$.
\begin{lemma}\label{Lem3}
Let $D$ be a self-adjoint operator on $L^2(\m^d;\C^n)$ and suppose that $(I+D^2)^{-1/2}\in B(L^{p_*},L^p)$ for all $p\in(1,2]$. Let $k=d\big(\frac{1}{p}-\frac{1}{2}\big)$, then $(I+D^2)^{-k/2}\in B(L^p,L^2)$.
\end{lemma}
\begin{proof}
 Let $T=I+D^2$. By hypothesis we have  $||T^{\frac{-(1+is)}{2}}||_{B(L^{2_{*}},L^2)}<\infty$ and $||T^{\frac{-is}{2}}||_{B(L^{2},L^2)}$ $<\infty$. Thus, by applying Stein's Interpolation Theorem (Theorem 1, \cite{Stein}) we obtain that $||T^{-k/2}||_{B(L^q,L^2)}<\infty$ for  $k=d(\frac{1}{q}-\frac{1}{2})$ where $q\in[2_*,2]$. Now notice that $||T^{-(1+is)}||_{B(L^{2_{**}},L^2)}$ $<\infty$ and $||T^{-is}||_{B(L^2,L^2)}<\infty$ for all $s\in\m$. Again applying Stein's interpolation theorem we obtain $||T^{-k'}||_{B(L^r,L^2)}<\infty$ where  $k'=\frac{d}{2}(\frac{1}{r}-\frac{1}{2})$ for $r\in[2_{**},2]$. Replacing $k'$ by $k/2$ we obtain that $||T^{-k/2}||_{B(L^r,L^2)}<\infty$ for $k=d(\frac{1}{r}-\frac{1}{2})$. By induction, we obtain that for any $p\in[2_{(*)m},2]$ where $2_{(*)m}>1$ for some $m$ and  $k=d(\frac{1}{p}-\frac{1}{2})$ we have $||T^{-k/2}||_{B(L^p,L^2)}<\infty$. Thus for a given $p\in(1,2]$ there exists a $q>p$ such that $q_*=p$ and there exists an $m\in\mathbb{N}$ such that $2_{(*){m+1}}<q_*\leq2_{(*)m}$. Hence, by the above argument for $k=d(\frac{1}{p}-\frac{1}{2})$ we have $||T^{-k/2}||_{B(L^p,L^2)}<\infty$.
\end{proof}

\begin{theorem}\label{Thm3.1}
Let $D$ be a self-adjoint operator on $L^2(\m^d;\C^n)$ and let the $C_0$-group $(e^{i\xi D})_{\xi\in\m}$ generated by $iD$ have finite speed of propagation. Suppose also that for all $q\in(1,2]$,\; $(I+D^2)^{-1/2}\in B(L^{q_*},L^q)$, such that $q_*>1$. Let $p\in (1,\infty)$, then for all $a\in C_c^\infty(\m)$, $\beta= d\big|\frac{1}{p}-\frac{1}{2}\big|+1$ and for any positive integer $M>1+d\cdot\frac{p'}{2}$ where $\frac{1}{p}+\frac{1}{p'}=1$ we have,
\begin{equation}
    ||a(D)||_{B(L^p)}\lesssim\sum_{j=0}^{M}\bigg|\bigg|\mathcal{F}\big\{(1+|\cdot|^2)^{\beta/2}\bigg(\frac{d}{d(\cdot)}\bigg)^{(j)}a(\cdot)\big\}\bigg|\bigg|_\infty.
\end{equation}
\end{theorem}
\begin{proof}
Let $q\in(1,2)$. To prove the $L^p$ boundedness (for all $p\in(1,\infty)$) of the operator $a(D)$ for all $a\in C_c^\infty(\m)$ we shall start with the operator $b(D)(I+D^2)^{-\frac{k}{2}}$ for $b\in C_c^\infty(\m)$, where $k=d(\frac{1}{q}-\frac{1}{2})$.

Now for $u\in L^q(\m^d)$, consider $||b(D)(I+D^2)^{-\frac{k}{2}}u||_2$
\begin{align*}
    &=\Big|\Big|\int_\m\hat{b}(\xi)\exp(i\xi D)(I+D^2)^{-\frac{k}{2}}ud\xi\Big|\Big|_2, \text{ (by Phillips calculus) }\\
    &\lesssim\int_\m|\hat{b}(\xi)|\;||\exp(i\xi D)(I+D^2)^{-\frac{k}{2}}u||_2d\xi\\
    &\lesssim\int_\m|\hat{b}(\xi)|\;||(I+D^2)^{-\frac{k}{2}}u||_2d\xi\;\text{ (by uniform boundedness of }\exp(i\xi D) \text{ on }L^2)\\
    &\lesssim||u||_{q}||\hat{b}||_{L^1}\text{ (by Lemma \ref{Lem3})}\\
    &\lesssim||u||_{q}||(1+|\cdot|)^2\hat{b}||_\infty\\
    &\lesssim||u||_{q}\sum_{j=0}^2\bigg|\bigg|\mathcal{F}\bigg\{\bigg(\frac{d}{d(\cdot)}\bigg)^{(j)}b(\cdot)\bigg\}\bigg|\bigg|_\infty\\
    &\lesssim||u||_q\sum_{j=0}^2\bigg|\bigg|\mathcal{F}\bigg\{\bigg(\frac{d}{d(\cdot)}\bigg)^{(j)}\tilde{b}(\cdot)\bigg\}\bigg|\bigg|_\infty \text{ (where }\tilde{b}(x)=b(x)(1+x^2)^{1/2}).
\end{align*}
The reason behind writing in terms of $\tilde{b}(x)$ can be seen later.
We shall now see that the operator $b(D)(I+D^2)^{-\frac{k}{2}}$ is $L^2-L^2$ off diagonal bounded for any $k>0$, in particular for $k$ as above.

Let $E$ and $F$ be two Borel subsets of $\m^d$ and $u\in L^2(\m^d)$. Consider
\begin{align*}
   ||1_Eb(D)(I+D^2)^{-\frac{k}{2}}1_Fu||_2=||1_Eb_k(D)1_Fu||_2
\end{align*}
where $b_k(D)=b(D)(I+D^2)^{-\frac{k}{2}}$. We will now obtain the bound using the Phillips calculus.
Since the group generated by $iD$ has finite speed of propagation, therefore for arbitrary Borel set $E,F\subseteq\m^d$ there exists $c>0$ (independent of $E,F$ and $\xi$) such that 
\begin{equation*}
    1_E\exp(i\xi D)1_F=0, \text{ if } \frac{d(E,F)}{|\xi|}\geq c.\\
\end{equation*}
\noindent
For any closed cubes $E,F\subseteq\m^d,$ and $u\in L^2(\m^d;\C^n)$ such that $u$ is supported in $F$ consider $||1_Eb_k(D)1_Fu||_2$
\begin{align*}
    \hspace{1cm}&=\bigg|\bigg|\int_\m\hat{b}_k(\xi)1_E\exp(i\xi D)1_Fu d\xi\bigg|\bigg|_2\\
    &\lesssim\int_{-\infty}^{-\frac{d(E,F)}{c}}|\hat{b}_k(\xi)|d\xi ||1_Fu||_2+\int_{\frac{d(E,F)}{c}}^\infty|\hat{b}_k(\xi)|d\xi ||1_Fu||_2\\
    &\lesssim\int_{\frac{d(E,F)}{c}}^\infty|(1+c|\xi|)^{M}\hat{b}_k(\xi)|\frac{d\xi}{(1+c|\xi|)^{M}}||1_Fu||_2,\text{ for all }M>1\\
    &\lesssim\big|\big|(1+c|\cdot|)^{M}\hat{b}_k\big|\big|_{\infty}\int_{\frac{d(E,F)}{c}}^\infty\frac{d\xi}{(1+c|\xi|)^{M}} ||1_Fu||_2\\
    &\lesssim\big|\big|(1+c|\cdot|)^{M}\hat{b}_k\big|\big|_{\infty}(1+d(E,F))^{-M+1}||1_Fu||_2\\
    &\lesssim\sum_{j=0}^{M}\bigg|\bigg|\mathcal{F}\bigg\{\Big(\frac{d}{d(\cdot)}\Big)^{(j)}b_k(\cdot)\bigg\}\bigg|\bigg|_\infty(1+d(E,F))^{-M+1}||1_Fu||_2\\
    &\lesssim\sum_{j=0}^{M}\bigg|\bigg|\mathcal{F}\bigg\{\Big(\frac{d}{d(\cdot)}\Big)^{(j)}\bigg[\tilde{b}(\cdot)(1+(\cdot)^2)^{-(\frac{k+1}{2})}\bigg]\bigg\}\bigg|\bigg|_\infty (1+d(E,F))^{-M+1}||1_Fu||_2,\\
    &\simeq\sum_{j=0}^{M}\bigg|\bigg|\mathcal{F}\bigg\{\Big(\frac{d}{d(\cdot)}\Big)^{(j)}\tilde{b}(\cdot)\bigg(\sum_{l=0}^{M-j}\bigg(\frac{d}{d(\cdot)}\bigg)^{(l)}(1+(\cdot)^2)^{-(\frac{k+1}{2})}\bigg)\bigg\}\bigg|\bigg|_\infty\\
    &\hspace{1.5cm}(1+d(E,F))^{-M+1}||1_Fu||_2\\
    &\lesssim\sum_{j=0}^{M}\sum_{l=0}^{M-j}\bigg|\bigg|\mathcal{F}\bigg\{\Big(\frac{d}{d(\cdot)}\Big)^{(j)}\tilde{b}(\cdot)\bigg\}*\mathcal{F}\bigg\{ \bigg(\frac{d}{d(\cdot)}\bigg)^{(l)}g_k\bigg\}\bigg|\bigg|_\infty(1+d(E,F))^{-M+1}||1_Fu||_2\\
     &\;\;\;\;\text{ where }g_k(x)=(1+x^2)^{-(\frac{k+1}{2})}\\
    &\lesssim\sum_{j=0}^{M}\sum_{l=0}^{M-j}\bigg|\bigg|\mathcal{F}\bigg\{\Big(\frac{d}{d(\cdot)}\Big)^{(j)}\tilde{b}(\cdot)\bigg\}\bigg|\bigg|_\infty\bigg|\bigg|\mathcal{F}\bigg\{ \bigg(\frac{d}{d(\cdot)}\bigg)^{(l)}g_k\bigg\}\bigg|\bigg|_{L^1}(1+d(E,F))^{-M+1}||1_Fu||_2.
\end{align*}
\noindent
Thus,
\begin{align*}
    \big|\big|1_Eb(D)(I+D^2)^{-\frac{k}{2}}1_Fu\big|\big|_2\lesssim&\sum_{j=0}^{M}\Big|\Big|\mathcal{F}\Big\{\Big(\frac{d}{d(\cdot)}\Big)^{(j)}\tilde{b}(\cdot)\Big\}\Big|\Big|_\infty(1+d(E,F))^{-M+1}||1_Fu||_2,\; \forall M>1
\end{align*}
where, by simple computation, for all $0\leq l\leq M$  we have
\begin{align*}
\bigg|\bigg|\mathcal{F}\bigg\{\bigg(\frac{d}{d(\cdot)}\bigg)^{(l)}g_k\bigg\}\bigg|\bigg|_{L^1} &=\bigg|\bigg|(1+|\cdot|)^{-1}(1+|\cdot|)\mathcal{F}\bigg\{\bigg(\frac{d}{d(\cdot)}\bigg)^{(l)}g_k\bigg\}\bigg|\bigg|_{L^1}\\
    &\leq||(1+|\cdot|)^{-1}||_2\bigg|\bigg|(1+|\cdot|)\mathcal{F}\bigg\{\bigg(\frac{d}{d(\cdot)}\bigg)^{(l)}g_k\bigg\}\bigg|\bigg|_2\\
    &\lesssim\bigg|\bigg|\bigg(\frac{d}{d(\cdot)}\bigg)^{(l)}g_k\bigg|\bigg|_2+\bigg|\bigg|\bigg(\frac{d}{d(\cdot)}\bigg)^{(l+1)}g_k\bigg|\bigg|_2<\infty.
\end{align*}
By applying Riesz-Thorin Interpolation theorem (Theorem 2.1, \cite{SS}) on the operator $b(D)(I+D^2)^{-\frac{k}{2}}$ we obtain the following off- diagonal bound from $L^p$ to $L^2$
\begin{align*}
    ||1_Eb(D)(I+D^2)^{-\frac{k}{2}}1_Fu||_2\lesssim&\sum_{j=0}^{M}\bigg|\bigg|\mathcal{F}\bigg\{\Big(\frac{d}{d(\cdot)}\Big)^{(j)}\tilde{b}(\cdot)\bigg\}\bigg|\bigg|_\infty\\
    &\;\;\;(1+d(E,F))^{-(M-1)\a}||1_Fu||_p
\end{align*}
for $p\in[q,2]$ and all $M>1+\frac{d}{\a}$, where $\a$ satisfies $\frac{1}{p}=\frac{1-\a}{q}+\frac{\a}{2}$ or  $\a=(\frac{1}{q}-\frac{1}{2})^{-1}\big(\frac{1}{q}-\frac{1}{p}\big)$. Since $q\in(1,2)$ is any arbitrary number so in the limiting case when $q$ approaches 1, we can say that the operator $b(D)(I+D^2)^{-\frac{k}{2}}$ has $L^p-L^2$ off-diagonal bound of order $>1$ if $M>1+d\cdot\frac{p'}{2}$ for $p\in(1,2]$ where $\frac{1}{p}+\frac{1}{p'}=1$.\\

\noindent
Let $\displaystyle a(x)=b(x)(1+x^2)^{-\frac{k}{2}}$ then 
\begin{align*}
    \big|\big|1_Ea(D)1_Fu\big|\big|_2&=\sum_{j=0}^{M}\bigg|\bigg|\mathcal{F}\bigg\{\Big(\frac{d}{d(\cdot)}\Big)^{(j)}\bigg[a(\cdot)(1+(\cdot)^2)^{\frac{k+1}{2}}\bigg]\bigg\}\bigg|\bigg|_\infty(1+d(E,F))^{-(M-1)\a}||1_Fu||_p\\
    &\lesssim\sum_{j=0}^{M}\sum_{l=0}^{M-j}\bigg|\bigg|\mathcal{F}\bigg\{(1+|\cdot|^2)^{\frac{(k+1)}{2}}\bigg[\Big(\frac{d}{d(\cdot)}\Big)^{(j)}a(\cdot)\bigg]\bigg[\frac{P_l(\cdot)}{(1+(\cdot)^2)^l}\bigg]\bigg\}\bigg|\bigg|_\infty\\ &\;\;\;\;\;(1+d(E,F))^{-(M-1)\a}||1_Fu||_p\text{ (where }P_l(x)\text{ is a polynomial of degree}< 2l,\\
    & \text{ and is obtained by taking } l^{th} \text{ derivative of } (1+x^2)^{\frac{k+1}{2}})\\
    &\lesssim\sum_{j=0}^{M}\sum_{l=0}^{M-j}\bigg|\bigg|\mathcal{F}\bigg\{(1+|\cdot|^2)^{\frac{k+1}{2}}\Big(\frac{d}{d(\cdot)}\Big)^{(j)}a(\cdot)\bigg\}*\mathcal{F}\bigg\{\frac{P_l(\cdot)}{(1+(\cdot)^2)^l}\bigg\}\bigg|\bigg|_\infty\\
    &\;\;\;\;\;(1+d(E,F))^{-(M-1)\a}||1_Fu||_p\\
    &\lesssim\sum_{j=0}^{M}\sum_{l=0}^{M-j}\bigg|\bigg|\mathcal{F}\Big\{(1+|\cdot|^2)^{\frac{\beta}{2}}\Big(\frac{d}{d(\cdot)}\Big)^{(j)}a(\cdot)\Big\}\bigg|\bigg|_\infty\bigg|\bigg|\mathcal{F}\bigg\{\frac{P_l(\cdot)}{(1+(\cdot)^2)^l}\bigg\}\bigg|\bigg|_{L^1}\\
    &\;\;\;\;\;(1+d(E,F))^{-(M-1)\a}||1_Fu||_p\\
    &\lesssim\sum_{j=0}^{M}\bigg|\bigg|\mathcal{F}\Big\{(1+|\cdot|^2)^{\frac{\beta}{2}}\Big(\frac{d}{d(\cdot)}\Big)^{(j)}a(\cdot)\Big\}\bigg|\bigg|_\infty(1+d(E,F))^{-(M-1)\a}||1_Fu||_p
\end{align*} 
where $\beta=k+1=d\big|\frac{1}{p}-\frac{1}{2}\big|+1$. The $L^1$ norm in the second last inequality is finite and can be computed in the same way as done for the functions $g_k$.
Thus, applying Lemma \ref{Lem 1} we obtain that for all $p\in(1,2]$
\begin{equation*}
    ||a(D)||_{B(L^p)}\lesssim\sum_{j=0}^{M}\bigg|\bigg|\mathcal{F}\bigg\{(1+|\cdot|^2)^{\frac{\beta}{2}}\bigg(\frac{d}{d(\cdot)}\bigg)^{(j)}a(\cdot)\bigg\}\bigg|\bigg|_\infty.
\end{equation*}
The result for $p\in[2,\infty)$ follows by duality.
\end{proof}


It can be seen from the following remark that we could obtain similar functional calculus through a simple argument if we knew that the group we have is polynomially bounded.
\begin{remark}
Let $X$ be an arbitrary Banach space and $D$ is an operator acting on $X$. If the group generated by $iD$ is polynomially bounded, say by a polynomial of order $M>0$, then a bound on the Phillips calculus of the operator $D$ is given by
\begin{align*}
    ||a(D)||_{B(X)}&=\bigg|\bigg|\int_\m\hat{a}(\xi)\exp(i\xi D)d\xi\bigg|\bigg|\\
    &=\bigg|\bigg|\int_\m(1+|\xi|)^M\hat{a}(\xi)\frac{\exp(i\xi D)}{(1+|\xi|)^M}d\xi\bigg|\bigg|\\
    &\lesssim||(1+|\cdot|)^M\hat{a}||_{L^1}\bigg|\bigg|\xi\mapsto\frac{\exp(i\xi D)}{(1+|\xi|)^M}\bigg|\bigg|_{L^\infty(B(X))}\\
    &\lesssim\sum_{j=0}^{M+2}\bigg|\bigg|\mathcal{F}\bigg\{\bigg(\frac{d}{d(\cdot)}\bigg)^{(j)}a(\cdot)\bigg\}\bigg|\bigg|_\infty
\end{align*}
\end{remark}
On $H_D^p$ instead of $L^p$, we now obtain a better bound without Sobolev embedding assumption.

\begin{theorem}\label{Thm3.2}
Let $1<p<\infty$ and let $D$ be a self-adjoint operator on $L^2(\m^d;\C^n)$. Suppose also that $(e^{i\xi D})_{\xi\in\m}$ has finite speed of propagation. Then for all $a\in C_c^\infty(\m)$ and $M>\frac{d}{\min(p,2)}$
\begin{equation}
    ||a(D)||_{B(H^p_D)}\lesssim\sum_{j=0}^{M+1}\bigg|\bigg|\o\mapsto|\o|^j\Big(\frac{d}{d\o}\Big)^{(j)}a(\o)\bigg|\bigg|_\infty.
\end{equation}
\end{theorem}
\begin{proof}
Let $u\in L^2(\m^d,\mathbb{C}^n)$. Fix $M=d$ and for $p\in(1,\infty)$ let $\psi\in\Psi_{d+1}^{d+1}(\S_\mu^0)$ where $\mu\in(0,\pi/2)$.
Consider $||a(D)u||_{H^p_D}$
\begin{align}
    &=||\Q_\psi(a(D)u)||_{T^{p,2}}\;(\text{by definition }[\ref{Def1}])\nonumber\\
    &=||(x,t)\mapsto\psi(tD)a(D)u(x)||_{T^{p,2}}\nonumber\\
    &\lesssim||(x,t)\mapsto a(D)\psi(tD)\psi(tD)u(x)||_{T^{p,2}}\label{A}\\
    &\lesssim C(a)||(x,t)\mapsto\psi(tD)u(x)||_{T^{p,2}}\label{B}
\end{align}
where $C(a)$(to be determined) is a constant that depends on $a$ and the inequality \eqref{B} holds by Lemma \ref{Lem 2}. For inequality \eqref{A}, we see by Theorem \ref{Thm2.1} and its proof (Theorem 7.10, \cite{MR2448989}) that we can replace $\psi$ by $\psi^2$ as they belong to the same class. To obtain the constant $C(a)$ we shall see that the operator $(a\psi_t)(D)$ has $L^2-L^2$ off diagonal bound of order $M>\frac{d}{\min\{p,2\}}$ where $\psi_t(D)=\psi(tD)$. Since $(e^{i\xi D})_{\xi\in\m}$ has finite propagation speed therefore for arbitrary Borel sets $E,F\subseteq\m^d$ there exists a $c>0$ such that $1_E\exp(i\xi D)1_F=0,\;\text{if }\frac{d(E,F)}{|\xi|}\geq c$. Let $E$ and $F$ be two Borel subsets of $\m^d$ and $M=d$. Consider $||1_E(a\psi_t)(D)1_Fu||_2$ 
\begin{align*}
    &=\bigg|\bigg|\int_\m\widehat{a\psi_t}(\xi)1_E\exp(i\xi D)1_Fud\xi\bigg|\bigg|_2\text{ (by Phillips calculus)}\\
    &\lesssim||\xi\mapsto(c\xi)^{M+1}\widehat{a\psi_t}(\xi)||_\infty\int_{\frac{d(E,F)}{c}}^\infty\frac{d\xi}{(c\xi)^{M+1}}||1_Fu||_2.
\end{align*}
Now consider $||\xi\mapsto(c\xi)^{M+1}\widehat{a\psi_t}(\xi)||_\infty$
\begin{align*}
    &\lesssim\Big|\Big|\o\mapsto\frac{d^{M+1}}{d\o^{M+1}}a\psi_t(\o)\Big|\Big|_{L^1}\\
    &\lesssim\sum_{j=0}^{M+1}\Big|\Big|\o\mapsto|\o|^j \Big(\frac{d}{d\o}\Big)^{(j)}a(\o)\Big|\Big|_\infty \int_\m\Big|t^{M+1-j}\psi^{(M+1-j)}(t\o)\frac{d\o}{|\o|^j}\Big|\\
    &\lesssim\sum_{j=0}^{M+1}\Big|\Big|\o\mapsto|\o|^j \Big(\frac{d}{d\o}\Big)^{(j)}a(\o)\Big|\Big|_\infty |t|^M\int_\m\Big|\psi^{(M+1-j)}(t\o)\frac{td\o}{(t|\o|))^j}\Big|\\
    &\lesssim\sum_{j=0}^{M+1}\Big|\Big|\o\mapsto|\o|^j \Big(\frac{d}{d\o}\Big)^{(j)}a(\o)\Big|\Big|_\infty |t|^M\int_\m\Big|\psi^{(M+1-j)}(\o)\frac{d\o}{(t|\o|)^j}\Big|\\
    &\hspace{7cm}\text{(by change of variables)}\\
    &\lesssim t^M\sum_{j=0}^{M+1}\Big|\Big|\o\mapsto|\o|^j \Big(\frac{d}{d\o}\Big)^{(j)}a(\o)\Big|\Big|_\infty, \text{ since }t\in\m^+
\end{align*}
where the integral $\displaystyle\int_\m\Big|\psi^{(M+1-j)}(\o)\frac{d\o}{(t|\o|)^j}\Big|$ is finite since for $\psi\in \Psi_{d+1}^{d+1}(\S^0_{\mu})$ the derivatives $\psi^{(j)}$ lie in the class $\Psi_{d+1-j}^{d+1}(\S^0_{\mu})$ for $j=0,1,...,M+1$ or we can say that $\eta^j\psi^{(j)}(\eta)\in\Psi_{d+1}^{d+1}(\S^0_{\mu})$ and hence by simple computation we obtain that the integral is uniformly bounded in $t$.
Thus,\\ $\displaystyle||1_E(a\psi_t)(D)1_Fu||_2\lesssim\Big(\frac{d(E,F)}{t}\Big)^{-M}\sum_{j=0}^{M+1}\Big|\Big|\o\mapsto|\o|^j \Big(\frac{d}{d\o}\Big)^{(j)}a(\o)\Big|\Big|_\infty||1_Fu||_2$.\\
Now to get the off diagonal bound consider two cases
\begin{itemize}
    \item[1.] If $\frac{d(E,F)}{t}\geq c\geq1$ then by some algebraic computation we can see that\\
                   $\Big(\frac{d(E,F)}{t}\Big)^{-M}\lesssim\Big(1+\frac{d(E,F)}{t}\Big)^{-M}.$
    \item[2.] If $1>\frac{d(E,F)}{t}\geq c$ then $||1_E(a\psi_t)(D)1_F||_{B(L^2)}$ is uniformly bounded by $||a||_\infty$.
\end{itemize}
Thus $(a\psi_t)(D)$ has $L^2-L^2$ off diagonal bound and hence, 
\begin{align*}
||a(D)u||_{H^p_D}
    &\lesssim\sum_{j=0}^{M+1}\Big|\Big|\o\mapsto|\o|^j \Big(\frac{d}{d\o}\Big)^{(j)}a(\o)\Big|\Big|_\infty||(x,t)\mapsto\psi(tD)u(x)||_{T^{p,2}}\\
    &\lesssim\sum_{j=0}^{M+1}\Big|\Big|\o\mapsto|\o|^j \Big(\frac{d}{d\o}\Big)^{(j)}a(\o)\Big|\Big|_\infty||\Q_\psi u||_{T^{p,2}}\\
    &\simeq\sum_{j=0}^{M+1}\Big|\Big|\o\mapsto|\o|^j \Big(\frac{d}{d\o}\Big)^{(j)}a(\o)\Big|\Big|_\infty||u||_{H^p_D}
\end{align*}
where the last equivalence holds by definition.
\end{proof}

The following Corollary is a specific case of the above Theorem for the operator $\Pi_B$ given in Definition \ref{Def8}.
\begin{corollary}\label{Cor1}
Let $p\in(p_H,p^H)$ as given in definition \ref{Def8}. If $D=\Pi_B$ in the above theorem is such that it is self-adjoint then for all $a\in C_c^\infty(\m)$ we have
\begin{equation}
    ||a(D)||_{B(L^p)}\lesssim\sum_{j=0}^{M+1}\Big|\Big|\o\mapsto|\o|^j \Big(\frac{d}{d\o}\Big)^{(j)}a(\o)\Big|\Big|_\infty.
\end{equation}
\end{corollary}
\begin{proof}
Using Corollary 3.2, \cite{FMP} we obtain that $H^P_{\Pi_B}=\overline{\R_p(\Pi_B)}$ for $p\in(p_H,p^H)$ and since $\Pi_B$ is bisectorial in $L^p(\m^d,\mathbb{C}^n)$ for $p\in(1,\infty)$ we know that $L^p(\m^d,\mathbb{C}^n)=\mathcal{N}_p(\Pi_B)\oplus\overline{\R_p(\Pi_B)}$. Thus for all $p\in(p_H,p^H),\;\mathcal{N}_p(\Pi_B)\oplus H^p_{\Pi_B}= L^p(\m^d,\mathbb{C}^n)$.
\end{proof}

Finally, we show that for operators that generate group with finite-speed of propagation on $L^2$, Kriegler-Weis result (Theorem \ref{Thm2.3}) on H{\"o}rmander calculus can be improved to require only boundedness (rather than R-boundedness) assumption.

\begin{theorem}\label{Thm3.3}
Let $D$ be a self-adjoint operator on $L^2(\m^d;\C^n)$ such that the group $(e^{itD})_{t\in\m}$ has finite speed of propagation. Suppose also that the following set is $L^{q_*}-L^q$ bounded for $q\in(1,2]$ and $\a>\frac{1}{2}$
\begin{equation*}
    (I)_\a:=\bigg\{\bigg(\frac{\pi}{2}-|\t|\bigg)^\alpha \sqrt{t}\exp(-te^{i\t} D^2):t>0,\t\in\bigg(\frac{-\pi}{2},\frac{\pi}{2}\bigg)\bigg\}.
\end{equation*}
Then the set
\begin{equation*}
    (S)_\alpha:=\bigg\{\bigg(\frac{\pi}{2}-|\t|\bigg)^\alpha \exp(-te^{i\t} D^2):t>0,\t\in\bigg(\frac{-\pi}{2},\frac{\pi}{2}\bigg)\bigg\}
\end{equation*}
is R-bounded for some $\a>0$(need not be same as above) and the operator $D^2$ has a R-bounded H{\"o}rmander calculus $\mathcal{H}^\gamma_2$ for $\gamma>\a+\frac{1}{2}$ on $L^q(\m^d;\C^n)$, $q\in(1,\infty)$.
\end{theorem}

\begin{proof}
We first prove this result for $q\in(1,2]$ and then use duality to prove it for all $q\in (1,\infty)$. Now for the former case we shall start proving the boundedness of the set $(S)_\a$ for $q\in(2_*,2]$ which requires the following steps.\\

\noindent
\textbf{Step 1}: $L^2-L^2$ boundedness of the set
\begin{equation*}
    (J)_\a:=\bigg\{\bigg(\frac{\pi}{2}-|\arg z|\bigg)^\a e^{\frac{d(E,F)^2}{c^2}Re(\frac{1}{z})}1_E\exp(-zD^2)1_F:z\in\Sigma_\t,\t\in(\frac{-\pi}{2},\frac{\pi}{2})\bigg\}
\end{equation*}
where $E,F\subset\m^d$ are arbitrary Borel sets, $c>0$ is a constant independent of $E$ and $F$, and $\a>\frac{1}{2}$. 

We shall prove this boundedness by using Phillips calculus. Since $iD$ generates a bounded $C_0$-group $\exp(itD)$ on $L^2(\m^d)$, by Phillips calculus we have
\begin{equation*}
    \Phi_g(D)u:=\int_{-\infty}^\infty g(t)\exp(itD)u dt, \; g\in L^1(\m), u\in L^2(\m^d).
\end{equation*}
Let $\Phi_g(D)=\exp(-zD^2)$ then $g(t)=\frac{e^{-t^2/4z}}{\sqrt{2}\sqrt{z}}$, for $Re(z)>0$. Since the group $\exp(itD)$ is bounded and has finite speed of propagation, therefore for arbitrary Borel sets $E,F\subset\m^d$ there exists $c>0$ (independent of $E,F$ and $t$) such that
\begin{equation*}
    1_E\exp(itD)(1_Fu)=0,\;\text{ if } \frac{d(E,F)}{|t|}\geq c.
\end{equation*}
Consider\; $||1_E\exp(-zD^2)(1_Fu)||_2$
\begin{align*}
    &=\bigg|\bigg|\int_{-\infty}^\infty \frac{e^{-t^2/4z}}{\sqrt{2}\sqrt{z}} 1_E\exp(itD)(1_Fu)dt\bigg|\bigg|_2\\
    &\lesssim\int_{-\infty}^\infty \frac{e^{\frac{-t^2}{4} Re(\frac{1}{z})}}{|\sqrt{z}|}||1_E\exp(itD)(1_Fu)||_2dt\\
    &=\int_{-\infty}^{-\frac{d(E,F)}{c}}\frac{e^{\frac{-t^2}{4} Re(\frac{1}{z})}}{|\sqrt{z}|}||1_E\exp(itD)(1_F u)||_2dt+\int_{\frac{d(E,F)}{c}}^\infty \frac{e^{\frac{-t^2}{4} Re(\frac{1}{z})}}{|\sqrt{z}|}||1_E\exp(itD)(1_F u)||_2dt\\
    &\simeq\int_{\frac{d(E,F)}{c}}^\infty \frac{e^{\frac{-t^2}{4} Re(\frac{1}{z})}}{|\sqrt{z}|}||1_E\exp(itD)(1_F u)||_2dt\\
    &\lesssim||1_Fu||_2 \frac{1}{\sqrt{\cos\t}}e^{-\frac{d(E,F)^2}{4c^2}Re(\frac{1}{z})}
\end{align*}
where the last inequality is obtained by change of variables, assuming $z=re^{i\t}$.
Therefore $||e^{\frac{d(E,F)^2}{4c^2}Re(\frac{1}{z})}1_E\exp(-zD^2)(1_F u)||_2\lesssim\frac{1}{(\frac{\pi}{2}-|\arg z|)^{1/2}}||1_Fu||_2$. Also for $\a>0$ we get,
\begin{equation*}
    \bigg|\bigg|\bigg(\frac{\pi}{2}-|\arg z|\bigg)^{\a+1/2}e^{\frac{d(E,F)^2}{4c^2}Re(\frac{1}{z})}1_E\exp(-zD^2)(1_F u)\bigg|\bigg|_2\lesssim||1_Fu||_2.
\end{equation*}
In other words, 
\begin{equation*}
   \bigg|\bigg|\bigg(\frac{\pi}{2}-|\arg z|\bigg)^{\a}e^{\frac{d(E,F)^2}{4c^2}Re(\frac{1}{z})}1_E\exp(-zD^2)(1_F u)\bigg|\bigg|_2\lesssim||1_Fu||_2, \text{ for }\a>\frac{1}{2}.
\end{equation*}
\noindent
Thus the set $(J)_\a$ is bounded on $L^2$ for $\a>\frac{1}{2}$.\\

\noindent
\textbf{Step 2}: $L^{2_*}-L^2$ boundedness of $(I)_\a$ and $L^2-L^2$ boundedness of $(J)_\a$ implies $L^q-L^2$ boundedness of $(K)^\a_\beta$ for $q\in[2_*,2]$ where $(K)^\a_\beta$ is defined as the set
\begin{equation*}
   (K)^\a_\beta:=\bigg\{\bigg(\frac{\pi}{2}-|\arg z|\bigg)^\a (Re\;z)^{\frac{\beta}{2}} e^{\frac{d(E,F)^2}{4c^2}Re(\frac{1}{z})(1-\beta)}1_E\exp(-zD^2)1_F:Re(z)>0\bigg\}
\end{equation*}
where $Re(\beta)\in[0,1)$ satisfies $\frac{1}{q}=\frac{Re(\beta)}{2_*}+\frac{1-Re(\beta)}{2}$ and $\a>\frac{1}{2}$. To prove this we shall use Stein's Interpolation Theorem once again. Let $s\in\m$ and consider the following operators (of the form $T_{\beta,z}\in(K)^\a_\beta$)
\begin{align*}
    &T_{1+is,z}:=\bigg(\frac{\pi}{2}-|\arg z|\bigg)^\a (Re\;z)^{(1+is)/2}e^{\frac{d(E,F)^2}{4c^2}Re(\frac{1}{z})(is)}1_E\exp(-zD^2)1_F\\
    &T_{0+is,z}:=\bigg(\frac{\pi}{2}-|\arg z|\bigg)^\a (Re\;z)^{(is)/2} e^{\frac{d(E,F)^2}{4c^2}Re(\frac{1}{z})(1-is)}1_E\exp(-zD^2)1_F
\end{align*}
By the proof in Step 1 and the assumption of the Theorem we have $T_{\beta,z}\in B(L^2,L^2)$ for $Re(\beta)=0$ and $T_{\beta,z}\in B(L^{2_*},L^2)$ for $Re(\beta)=1$, respectively. Now by Stein's Interpolation Theorem for $Re\;\beta\in[0,1)$ the operator $T_{\beta,z}(\in (K)^\a_\beta)$ belongs to $B(L^q,L^2)$ for $q\in(2_*,2]$ and $Re\;\beta=\frac{d}{q}-\frac{d}{2}$.
That is, for $\a>\frac{1}{2}$ and $z=te^{i\t}$ where $\t\in(\frac{-\pi}{2},\frac{\pi}{2})$ and $t>0$ we have,
\begin{equation*}
    \bigg|\bigg|\bigg(\frac{\pi}{2}-|\t|\bigg)^\a (t\cos\t)^{\frac{\beta}{2}}e^{\frac{d(E,F)^2}{4c^2t}(1-\beta)\cos\t}1_E\exp(-te^{i\t}D^2)1_F u\bigg|\bigg|_2\lesssim||1_Fu||_q.
\end{equation*}

\noindent
\textbf{Step 3:}
Assume now that $E=B(x,\sqrt{t})$ and $F=B(x,(k+1)\sqrt{t})\setminus B(x,k\sqrt{t})$, where $x\in\m^d$, $k\geq0$ and $t>0$.  Then $d(E,F)=(k-1)\sqrt{t}$ for $k\geq1$.
Applying Theorem \ref{Thm2.2} by Kunstmann and using
\begin{align*}
    \bigg|\bigg|\bigg(\frac{\pi}{2}-|\t|\bigg)^\a 1_E\exp(-te^{i\t}D^2)1_F \bigg|\bigg|_{q\to2}&\lesssim (t\cos\t)^{-Re\;\beta/2}e^{-\frac{d(E,F)^2}{4c^2t}\cos(\t)Re(1-\beta)}\\
    &\lesssim\frac{|B(x,\sqrt{t})|^{-\big(\frac{1}{q}-\frac{1}{2}\big)}}{(\cos\t)^{\frac{Re\;\beta}{2}}}e^{-c'(k-1)^2Re(1-\beta)\cos\t}\\
    &\lesssim\frac{|B(x,\sqrt{t})|^{-\big(\frac{1}{q}-\frac{1}{2}\big)}}{(\cos\t)^{\frac{\delta}{2}+\frac{Re\;\beta}{2}}}\frac{1}{(1+c'(k-1)^2Re(1-\beta))^{\delta/2}},\\
    &\hspace{3cm}\text{ for } \delta>\frac{d}{q}+\frac{1}{q'}\\
    &\lesssim\frac{|B(x,\sqrt{t})|^{-\big(\frac{1}{q}-\frac{1}{2}\big)}}{(\cos\t)^{\frac{\delta}{2}+\frac{Re\;\beta}{2}}}\frac{c_\delta}{(1+k)^\delta},
\end{align*}
that is, for $\a>\frac{1}{2}$ and $\delta>\frac{d}{q}+\frac{1}{q'}$ we have,
\begin{equation*}
     \bigg|\bigg|\bigg(\frac{\pi}{2}-|\t|\bigg)^{\a+\frac{\delta}{2}+\frac{Re\;\beta}{2}} 1_E\exp(-te^{i\t}D^2)1_F \bigg|\bigg|_{q\to2}\lesssim|B(x,\sqrt{t})|^{-\big(\frac{1}{q}-\frac{1}{2}\big)}\frac{c_\delta}{(1+k)^\delta}.
\end{equation*}
Or we can write,
\begin{equation*}
     \bigg|\bigg|\bigg(\frac{\pi}{2}-|\t|\bigg)^{\a} 1_E\exp(-te^{i\t}D^2)1_F \bigg|\bigg|_{q\to2}\lesssim|B(x,\sqrt{t})|^{-\big(\frac{1}{q}-\frac{1}{2}\big)}\frac{c_\delta}{(1+k)^\delta}
\end{equation*}
for  $\a>\frac{1}{2}+\frac{\delta}{2}+\frac{d}{2}\big(\frac{1}{q}-\frac{1}{2}\big)$.
Thus, by Theorem \ref{Thm2.2} we obtain that the set $(S)_\a$ is R-bounded on $L^q$ for $q\in(2_*,2]$. Hence the operator $D^2$ has R-bounded $\mathcal{H}^{\gamma}_2$ calculus for $\gamma>\a+\frac{1}{2}$ on $L^q$ for $q\in(2_*,2]$ by Theorem \ref{Thm2.3}.

\textbf{Step 4:} We now prove the R-boundedness of $(S)_\a$ on $L^q$ for $q\in(p_*,p]$ where $p\in(2_*,2]$ and then take it down to the least Sobolev exponents below 2 which remains greater than 1 so that we can prove the R-boundedness for all $p\in(1,2]$.

For this we already have $L^{p_*}-L^p$ boundedness of $(I)_\a$ by hypothesis. Using this and the fact that the set
$(K)^\a_\beta$ is $L^p-L^2$ bounded for $p\in(2_*,2]$ we obtain that 
\begin{equation*}
    \bigg\{\bigg(\frac{\pi}{2}-|\arg z|\bigg)^{2\a}(Re\;z)^{\frac{1+\beta}{2}} e^{\frac{d(E,F)^2}{c^2}Re(\frac{1}{z})(1-\beta)}1_E\exp(-2zD^2)1_F:z\in\Sigma_\t, |\t|<\frac{\pi}{2}\bigg\}
\end{equation*}
is $L^{p_*}-L^2$ bounded for $\a>1$. Interpolating this with the $L^2-L^2$ bound of 
\begin{equation*}
  \bigg\{\bigg(\frac{\pi}{2}-|\arg z|\bigg)^{2\a}\exp\bigg(2\frac{d(E,F)^2}{c^2}Re(\frac{1}{z})\bigg)1_E\exp(-zD^2)1_F:z\in\Sigma_\t, |\t|<\frac{\pi}{2}\bigg\}
\end{equation*}
we get $L^q-L^2$ boundedness of
\begin{equation*}
 \bigg\{\bigg(\frac{\pi}{2}-|\arg z|\bigg)^{2\a}\frac{ e^{\frac{d(E,F)^2}{c^2}Re(\frac{1}{z})(2-\eta(1+\beta))}}{(Re\;z)^{-\frac{(1+\beta)\eta}{2}}}1_E\exp(-2zD^2)1_F:z\in\Sigma_\t, |\t|<\frac{\pi}{2}\bigg\}
 \end{equation*}
 for $q\in(p_*,2]$ and $\a>1$ and where $\eta$ satisfies $\frac{1}{q}=\frac{\eta}{p_*}+\frac{1-\eta}{2}$. Now proceeding similarly as in \textbf{Step 3} and using Theorem \ref{Thm2.2} we get that the set $\Big\{\big(\frac{\pi}{2}-|\arg z|\big)^{2\a}\exp(-2zD^2)\Big\}$ has $L^q-L^q$ R-bound, that is, the set $(S)_{2\a}$ has $L^q-L^q$ R-bound for $q\in(2_{**},2_*]$ and $2\a\geq1+\frac{\delta}{2}+\frac{(1+Re\;\beta)Re\;\eta}{2}$, where $Re\;\beta,Re\;\eta\in(0,1)$ or $2\a\geq1+\frac{\delta}{2}+d\big(\frac{1}{q}-\frac{1}{2}\big)$. This implies that the operator $D^2$ has H{\"o}rmander calculus $\mathcal{H}^\gamma_2$ for $\gamma>2\a+\frac{1}{2}$ on $L^q$ for $q\in(2_{**},2]$.
 
 By induction we can go down to the Sobolev exponent $2_{(*)m}$ as long as it is greater than 1 for some $m$. If $2_{(*)m}\leq1$ then for any $r=1+\e,\;(\e>0)$ there exists a $p>2_{(*){m-1}}$ such that $r=p_*$. Then by same steps as done previously we will have $L^q-L^q$ R-bound for $(S)_{\a'}$ for $q\in(p_*,p]$ and some $\a'>0$. Therefore we get that $(S)_\a$ is R-bounded on $L^q$ for $q\in(1,2]$ and some $\a>0$ such that $\a\geq c_1+\frac{\delta}{2}+c_2d(\frac{1}{q}-\frac{1}{2})$, where $c_1$, $c_2$ and $\delta$ depends on the choice of $q$. By Kriegler-Weis result (Theorem \ref{Thm2.3}) we thus obtain that $D^2$ has $\mathcal{H}^\gamma_2$ calculus on $L^q$ for $q\in(1,2]$ and $\gamma>\a+\frac{1}{2}$ for some $\a>0$.
 
 For $q\in(2,\infty)$ we use duality of R-bounded sets (Proposition 8.4.1, \cite{ABS}) and the fact that the operator $D^2$ is self-adjoint. Since $\exp(-zD^2)\in B(L^q)$,
 $\langle\exp(-zD^2)u,v\rangle=\langle u, \exp(-\bar{z}D^2)v\rangle$ for all
 $u\in L^2\cap L^q,\;v\in L^2\cap L^{q'}$. Thus, $\exp(-zD^2)^*v=\exp(-\bar{z}D^2)v,\;\forall\;v\in L^2\cap L^{q'}$. By density, $\exp(-zD^2)^*=\exp(-\bar{z}D^2)\text{ on }B(L^{q'})$. Thus, by duality we get $(S)_\a$ is R-bounded on $L^{q'}$ where $\frac{1}{q}+\frac{1}{q'}=1$. Hence, the operator $D^2$ has H{\"o}rmander calculus $\mathcal{H}^\gamma_2$ on $L^q$ for $q\in(1,\infty).$
\end{proof}

One of the useful applications of this theorem is that the square of a perturbed Hodge-Dirac operator $\Pi_B$ also has a R-bounded H{\"o}rmander calculus, as can be seen in the following corollary.

\begin{corollary}\label{Cor2}
Let $D=\Pi_B$ in the above theorem be self-adjoint. Then $\Pi_B^2$ has a R-bounded H{\"o}rmander calculus on $L^p(\m^d;\C^n)$ for $p\in(p_H,p^H)$.
\end{corollary}
\begin{proof}
Using Lemma 7.1, \cite{FMP} we can prove the $L^{q_*}-L^q$ boundedness of the set $(I)_\a$ for $q\in(1,2]$ required in the above theorem for the operator $\Pi_B^2$. The result then follows from the proof of the above theorem.
\end{proof}

\begin{remark}
We can, in fact, obtain the R-bounded H{\"o}rmander calculus on $L^p$ for $p\in(1,\infty)$ for those perturbed Hodge-Dirac operators for which $p_H=1$ and $p^H=\infty$ (for example, with real or smooth coefficients).
\end{remark}
\endgroup

\bibliographystyle{abbrv}
\bibliography{bibliography}
\end{document}